\documentclass[11pt,a4paper,twoside,reqno]{amsart}

\usepackage{graphicx}
\usepackage{epsfig,
pstricks,
curves}
\usepackage{float}
\usepackage{amsfonts}
\usepackage{amssymb}

\usepackage{fontenc}
\usepackage{amsmath}
\usepackage{latexsym}

\newcommand{\N}{{\mathbb N}}

\newcommand{\Z}{{\mathbb Z}}
\newcommand{\R}{{\mathbb R}}

\newcommand{\kk}{\Bbbk}

\newcommand{\M}{\mathbf{M}}

\newcommand{\A}{\mathcal{W}_{n-1}}
\newcommand{\ZZ}{\mathcal{Z}}

\newcommand{\inj}{{\rm inj}}

\newcommand{\vol}{{\rm vol}}

\newcommand{\length}{{\rm length}}

\newcommand{\cf}{{\it cf.}}
\newcommand{\ie}{{\it i.e.}}

\numberwithin{equation}{section}

\newtheorem{theorem}{Theorem}[section]
\newtheorem{proposition}[theorem]{Proposition}
\newtheorem{corollary}[theorem]{Corollary}
\newtheorem{lemma}[theorem]{Lemma}

\theoremstyle{definition}
\newtheorem{definition}[theorem]{Definition}

\newtheorem{remark}[theorem]{Remark}

\long\def\forget#1\forgotten{} %

\title[Volume of minimal hypersurfaces]{Volume of minimal hypersurfaces in manifolds with nonnegative Ricci curvature}

\author[S.~Sabourau]{St\'ephane Sabourau}

\address{Universit\'e Paris-Est,
Laboratoire d'Analyse et Math\'ematiques Appliqu\'ees (UMR 8050), 
UPEC, UPEMLV, CNRS, F-94010, Cr\'eteil, France}

\email{stephane.sabourau@u-pec.fr}



\begin{document}

\begin{abstract}
We establish a min-max estimate on the volume width of a closed Riemannian manifold with nonnegative Ricci curvature.
More precisely, we show that every closed Riemannian manifold with nonnegative Ricci curvature admits a PL Morse function whose level set volume is bounded in terms of the volume of the manifold.
As a consequence of this sweep-out estimate, there exists an embedded, closed (possibly singular) minimal hypersurface whose volume is bounded in terms of the volume of the manifold.
\end{abstract}

\maketitle


\section{Introduction}

In this article, we show the following result.

\begin{theorem} \label{theo:0}
Let $M$ be a closed Riemannian $n$-manifold with nonnegative Ricci curvature.
There exists an embedded closed minimal hypersurface~$S$ in~$M$ with a singular set of Hausdorff dimension at most~$n-8$ such that
\[
\vol_{n-1}(S) \leq C_n \, \vol_n(M)^{\frac{n-1}{n}}
\]
where $C_n$ is an explicit positive constant depending only on~$n$.
\end{theorem}


Observe that both the Ricci curvature condition and the inequality are scale invariant in this theorem.
We do not know whether the curvature condition can be dropped in dimension greater than two.
In dimension two, it is the case according to~\cite[Theorem~5.3.1]{BZ}, \cite{heb} and~\cite{cro}: every closed Riemannian surface has a closed geodesic of length bounded from above in terms of the area of the surface.
Thus, Theorem~\ref{theo:0} can be viewed as a partial generalization of this result on closed geodesics on surfaces in terms of minimal hypersurfaces.

Another way to generalize this result would be to find an upper bound on the length of the shortest closed geodesic in a closed Riemannian manifold in terms of the volume.
Except for essential manifolds, where systolic inequalities hold~\cite{gro83}, and the two-sphere~\cite{cro}, this question is completely open, even on manifolds with nonnegative Ricci curvature.

The condition on the size of the singular set is customary in this context.
To our knowledge, it is still unknown if there exists a closed \emph{smooth} minimal hypersurface in every closed Riemannian manifold even with nonnegative Ricci curvature.
The existence  of a minimal hypersurface in Theorem~\ref{theo:0} actually derives from a min-max principle due to Almgren-Pitts in geometric measure theory (see~\cite{pitts,SS81,CD03,DT}).
When the Ricci curvature is positive, the geometry of the min-max minimal hypersurface~$S$ has recently been described in~\cite{zhou}. 

In fact, Theorem~\ref{theo:0} is a consequence of the following sweep-out estimate.

\begin{theorem} \label{theo:A}
Let $M$ be a closed Riemannian $n$-manifold with nonnegative Ricci curvature.
There exists a PL Morse function $f:M \to \R$ such that
\[
\sup_t \vol_{n-1}  (f^{-1}(t)) \leq C_n \, \vol_n(M)^{\frac{n-1}{n}}
\]
where $C_n$ is an explicit positive constant depending only on~$n$.
\end{theorem}

Recall that a function $f:M \to \R$ is a PL Morse function if there exists a simplicial complex structure on~$M$ such that $f$ is linear on every simplex of~$M$ and takes pairwise distinct values at the vertices of~$M$.

As already mentioned, the sweep-out estimate of Theorem~\ref{theo:A} is stronger than Theorem~\ref{theo:0}.
A slightly more general version is given by Theorem~\ref{theo:main}.

Observe again that both the Ricci curvature condition and the inequality are scale invariant.
Contrary to the previous theorem where the question is open, this result fails without any curvature condition for every closed $n$-manifold with~$n \geq 3$ (see~Proposition~\ref{prop:cex}).
However, the curvature assumption can be dropped for closed Riemannian surfaces~\cite{BS10}.
In this case, the multiplicative constant depends on the genus of the surface (and it has to).

None of the multiplicative constants in Theorems~\ref{theo:0} or~\ref{theo:A} is optimal, even on two-dimensional convex spheres despite a local extremality result~\cite{bal10,sab10}.
For positively curved $3$-manifolds, related sharp upper bounds have been obtained in~\cite{MN12}.

When $M$ is a smooth convex hypersurface in a Euclidean $n$-space, the result of Theorem~\ref{theo:A} follows from~\cite{tre}.
A similar result also holds true for domains of the Euclidean $n$-space (see~\cite{fal80,guth07}).
Thus, our sweep-out estimate can be seen as a partial generalization of these results to a non-Euclidean setting.

Actually, Theorem~\ref{theo:A} can be formulated in terms of a min-max estimate by taking the infimum over all PL Morse functions of the supremum of the volume of their fibers.
Further consideration of min-max principles can be found in~\cite{gro83,gro88}.
Note that min-max processes are also involved in concentration phenomena. 
A related estimate holds for manifolds with nonnegative sectional curvature by replacing the volume of the fibers by the diameter (see~\cite{per95}).

As a consequence of our sweep-out estimate, we also derive the following isoperimetric inequality.

\begin{corollary} \label{coro:decomp}
There exists a positive constant~$C_n$ such that every closed Riemannian $n$-manifold~$M$ with nonnegative Ricci curvature decomposes into two connected domains with the same volume whose common boundary~$S$ satisfies
\[
\vol_{n-1}(S) \leq C_n \, \vol_n(M)^{\frac{n-1}{n}}.
\]
\end{corollary}

To conclude this introduction, we mention some recent applications of the min-max process on three-dimensional manifolds to emphasize its importance in current research.
(The most recent results mentioned below even appeared after this paper was submitted for publication.)
In~\cite{CM05}, T.~Colding and W.~Minicozzi established min-max estimates through the Ricci flow which allowed them to simplify some arguments in Perelman's proof of the Poincar\'e conjecture through  a more conceptual approach.
In~\cite{DP10} and~\cite{Ke}, the authors established genus bound for minimal surfaces constructed via min-max arguments on every closed Riemannian $3$-manifold.
Recently, F.~Marques and A.~Neves proved the Willmore conjecture~\cite{MN} and a conjecture of Yau about the existence of infinitely many minimal hypersurfaces in closed Riemannian manifolds with positive Ricci curvature~\cite{MN14} by using a min-max process \`a la Almgren-Pitts.

We also mention that P. Glynn-Adey and Y. Liokumovich recently placed a preprint on arxiv \cite{GAL} in which they establish Theorem~\ref{theo:0} among other results. 
The techniques seem somewhat different though they also rely on sweep-out estimates.

\medskip

Throughout this paper, a domain of a complete $n$-manifold with piecewise smooth boundary is an $n$-submanifold with piecewise smooth boundary.


\section{CW complex structure with quasi-convex cells of small size}

In this section, we introduce some key notions for the rest of the article and establish preliminary results which will be used later.

\medskip

The definition of a CW complex can be found in~\cite[Appendix]{hat} along with some related topological properties and applications.
We simply recall that every $k$-cell of a CW complex~$X$ is attached to the $(k-1)$-skeleton~$X^{(k-1)}$ of~$X$.
The map from this $k$-cell to~$X$ is referred to as the characteristic map of the $k$-cell.

We will denote by $B^n(\rho)$ the (open) $\rho$-ball centered at the origin of~$\R^n$.
Recall also that a $\lambda$-quasi-isometry between two metric space is a homeomorphism which is $\lambda$-Lipschitz and whose reciprocal map is also $\lambda$-Lipschitz.

\medskip

The geometry of a CW complex structure can be measured through the following notion.

\begin{definition} \label{def:cellstruc}
Let $\rho >0$ and $\lambda >1$.
A compact Riemannian $n$-manifold~$D$ with (possibly empty) piecewise smooth boundary admits a CW complex structure with quasi-convex cells of size~$(\rho,\lambda)$ if 
\begin{enumerate}
\item There exists a $\lambda$-quasi-isometry $\varphi:P \to D$ between a piecewise flat simplicial $n$-complex $P$ and~$D$. \label{i}
\item The simplicial $n$-complex~$P$ decomposes into $n$-cells $\Delta_k = \cup_i \, \sigma_{k,i}$ formed of a union of simplices~$\sigma_{k,i}$ such that each $n$-cell~$\Delta_k$ is $\lambda$-quasi-isometric to a convex Euclidean polyhedron~$\mathcal{E}_k$ lying in~$B^n(6 \rho)$ and containing~$B^n(2\rho)$.
This quasi-isometry is denoted by $\chi_k:\Delta_k \to \mathcal{E}_k$. \!\!\!\!\!\!\!\! \label{ii}
\item For every $k$, the composite $\varphi \circ \chi_k^{-1}:\mathcal{E}_k \to \varphi(\Delta_k)$ is $\lambda$-quasi-isometric. \label{iii}
\end{enumerate}
The decomposition of~$P$ into $n$-cells $\Delta_k$ gives rise to a CW complex structure on~$D$ where the characteristic maps are given by the quasi-isometries~$\varphi \circ \chi_k^{-1}$ described in~\eqref{iii}.
\end{definition}

For our purpose, we can think of~$D$ as an $n$-polyhedron whose $n$-faces are quasi-isometric to convex Euclidean polyhedra.
Despite the risk of confusion, we will often identify the $n$-cells of~$D$ with their images by the characteristic maps.

\medskip

The $n$-simplices composing the piecewise flat simplicial $n$-complex~$P$ are not necessarily regular as their edges may have different lengths.
However the $n$-cells of~$P$ have a uniform size, roughly the size of~$B^n(\rho)$, while still being almost convex.
(This is the reason why we introduce this definition.)
In particular, the volume of the $n$-cells is between $\lambda^{-n} \omega_n 2^n \rho^n$ and $\lambda^n \omega_n 6^n \rho^n$, where $\omega_n$ represents the volume of the unit ball centered at the origin of~$\R^n$.
Thus, if $N$ represents the number of $n$-cells in~$D$, we immediately obtain
\begin{equation} \label{eq:N}
N \lambda^{-n} \omega_n 2^n \rho^n \leq \vol_n(D) \leq N \lambda^n \omega_n 6^n \rho^n.
\end{equation}

The following result shows the existence of CW complex structures with quasi-convex cells of small size on closed Riemannian manifolds.

\begin{proposition} \label{prop:existence}
Let $D$ be a bounded domain of a complete Riemannian $n$-manifold~$M$.
For every $\rho > 0$ small enough, the domain~$D$ admits a CW complex structure with quasi-convex cells of size~$(\rho,\lambda_\rho)$ with $\lim_{\rho \to 0} \lambda_\rho = 1$.
\end{proposition}

\begin{proof}
Without loss of generality, we can assume that the boundary of~$D$ is smooth (otherwise we smooth it out).
Let $\rho \in (0,\frac{1}{100} \inj(M))$ with $\rho$ smaller than half the focal radius of~$\partial D$.
Let $\partial D_{= 2\rho} = \{ x \in D \mid d_M(x,\partial D) = 2\rho \}$ and  $\partial D_{\leq 2\rho} = \{ x \in D \mid d_M(x,\partial D) \leq 2\rho \}$.
Consider a maximal system of disjoint $2\rho$-balls of~$M$ centered in~$\partial D_{= 2\rho}$ and denote by~$x_i$ the centers of these balls.
Note that $B(x_i,2\rho)$ lies in~$D$.
Since the system of balls is maximal, the balls~$B(x_i,4\rho)$ cover~$\partial D_{= 2\rho}$.
As $\rho$ is smaller than half the focal radius of~$\partial D$, every point of~$\partial D$ is at distance~$2\rho$ from some point of~$\partial D_{=2 \rho}$.
Thus, the balls~$B(x_i,6 \rho)$ cover the $2\rho$-neighborhood~$\partial D_{\leq 2\rho}$ of~$\partial D$ in~$D$.


Now, we complete the collection of disjoint balls $B(x_i,2\rho)$ of~$M$ lying in~$D$ into a maximal system of disjoint $2\rho$-balls of~$M$ lying in~$D$.
We will still denote by~$x_i$ the centers of the balls thus obtained.
The Voronoi cell 
\[
V_i = \{ x \in D \mid d(x,x_i) \leq d(x,x_j) \mbox{ for every } j \neq i \}
\]
centered at~$x_i$ clearly contains the ball~$B(x_i,2\rho)$.
It also lies in the ball~$B(x_i,6 \rho)$.
Otherwise, we could find a point~$x \in D$ at distance at least~$6 \rho$ from any point~$x_j$.
From the previous paragraph, this implies that the point~$x$ is at distance at least~$\rho$ from~$\partial D$.
Thus, the ball~$B(x,2\rho)$ lies in~$D$ and is disjoint from the other balls~$B(x_j,2\rho)$.
This contradicts the construction of the points~$x_i$.

A consequence of these inclusions and the Gauss lemma is that the preimage~$\hat{V}_i$ of~$V_i$ by the exponential map~$\exp_{x_i}$ based at~$x_i$ lies in~$B^n(6 \rho)$ and contains~$B^n(2\rho)$.
This preimage is quasi-isometric to the convex polyhedral Voronoi cell~$\mathcal{E}_i$ formed of the vectors~$u$ of~$T_{x_i}M$ whose distance to the origin is less or equal to the distance from~$u$ to the preimages of the~$x_j$ under~$\exp_{x_i}$ for every $j \neq i$.
Furthermore, the factor of the quasi-isometry $V_i \to \hat{V}_i \to \mathcal{E}_i$ tends to~$1$ when $\rho$ goes to~$0$, \cf~Definition~\ref{def:cellstruc}, item~\eqref{iii}.
(In the following, all the quasi-isometries will satisfy this property.)

We would like to glue the convex Euclidean polyhedra~$\mathcal{E}_i$ together to construct a simplicial $n$-complex~$P$ quasi-isometric to~$D$.
Two adjacent Voronoi cells~$V_i$ and~$V_j$ in~$D$ meeting along an $(n-1)$-face give rise to two convex Euclidean polyhedra~$\mathcal{E}_i$ and~$\mathcal{E}_j$ along with two corresponding $(n-1)$-faces $e_{i,j}$ and~$e_{j,i}$ in~$\mathcal{E}_i$ and~$\mathcal{E}_j$.
If $e_{i,j}$ and~$e_{j,i}$ were isometric, we could glue the convex Euclidean polyhedra~$\mathcal{E}_i$ and~$\mathcal{E}_j$ together.
Unfortunately, this is not always the case.
Indeed, the $(n-1)$-faces $e_{i,j}$ and~$e_{j,i}$, though combinatorially equivalent, are only quasi-isometric.

To get around this technical problem, we consider a quasi-geodesic triangulation of the convex Euclidean polyhedra~$\mathcal{E}_i$ such that the induced triangulations on the Voronoi cells~$V_i$ in~$D$ agree on their common faces.
This gives rise to a triangulation~$\mathcal{T}$ of~$D$ (refining the CW complex structure of~$D$ given by the Voronoi cells) where each simplex~$\sigma$ is quasi-isometric to a Euclidean simplex~$\sigma'$. 

Now, we need to deform these Euclidean simplices~$\sigma'$ in order to glue them together.
Let $e$ be an edge of~$\mathcal{T}$.
Denote by $\sigma'_j$, where $j=1,\cdots,p$, the Euclidean simplices with an edge~$e'_j$ corresponding to~$e$.
Replace the edges~$e'_j$ of~$\sigma'_j$ with an edge of the same length equal to the average of the lengths of~$e'_1,\cdots,e'_p$.
As the edges~$e'_j$ are quasi-isometric to~$e$, the same goes for the new edges.

By applying this averaging argument to every edge of the triangulation~$\mathcal{T}$, we obtain new Euclidean simplices~$\sigma''$ quasi-isometric to the initial Euclidean simplices~$\sigma'$ and so to the simplex~$\sigma$.
Furthermore, the $(n-1)$-faces of these new Euclidean simplices~$\sigma''$ corresponding to common faces of adjacent simplices in~$D$ are isometric since the lengths of their edges agree.

Therefore, we can replace each simplex~$\sigma$ in~$D$ with a quasi-isometric Euclidean simplex~$\sigma''$ so that the resulting space is a simplicial $n$-complex~$P$ quasi-isometric to~$D$, \cf~Definition~\ref{def:cellstruc}, item~\eqref{ii}.
By construction, the simplices of~$P$ corresponding to the Euclidean simplices of the triangulation of~$\mathcal{E}_k$ form an $n$-cell~$\Delta_k$ quasi-isometric to~$\mathcal{E}_k$, \cf~Definition~\ref{def:cellstruc}, item~\eqref{ii}.
Since the factors of all these quasi-isometries tends to~$1$ when $\rho$ goes to~$0$, this yields the desired result.
\end{proof}

We will also need the following result.

\begin{lemma} \label{lem:proj}
Fix $\rho >0 $ and $\lambda>1$ close enough to~$1$.
Let $\Delta$ be an $n$-cell of a CW complex structure with quasi-convex cells of size~$(\rho,\lambda)$ on a bounded domain of a compact Riemannian $n$-manifold.
There exists a retraction 
\[
\pi_\Delta:\Delta \setminus B(\lambda \rho) \to \partial \Delta
\]
of~$\partial \Delta$ with Lipschitz constant at most~$36 \lambda^2$.
\end{lemma}

\begin{proof}
By definition, the $n$-cell $\Delta$ is $\lambda$-quasi-isometric to a convex Euclidean polyhedron~$\mathcal{E}$ lying in~$B^n(6 \rho)$ and containing~$B^n(2\rho)$.
The radial projection from $\mathcal{E} \setminus B^n(\rho)$ to~$\partial \mathcal{E}$ is Lipschitz.
A sharp upper bound on its Lipschitz constant is given by the limit from the right at $t=0$ of the length ratio $\frac{A'B'}{AB}$, where the points $A,A',B,B' \in \R^2$ are defined as follows, \cf~Figure~\ref{fig:proj}.
\begin{itemize}
\item The coordinates of $A$ and~$A'$ are $(\rho,0)$ and $(6 \rho,0)$.
\item The point~$B$ lies in the same vertical line as~$A$ and satisfies~$\widehat{AOB}=t$.
\item The point~$B'$ is the intersection point of $(OB)$ with the line tangent at~$A$ to the circle of radius~$\rho$ centered at the origin.
\end{itemize}

\begin{figure}[htbp] 
\def\svgwidth{8cm} 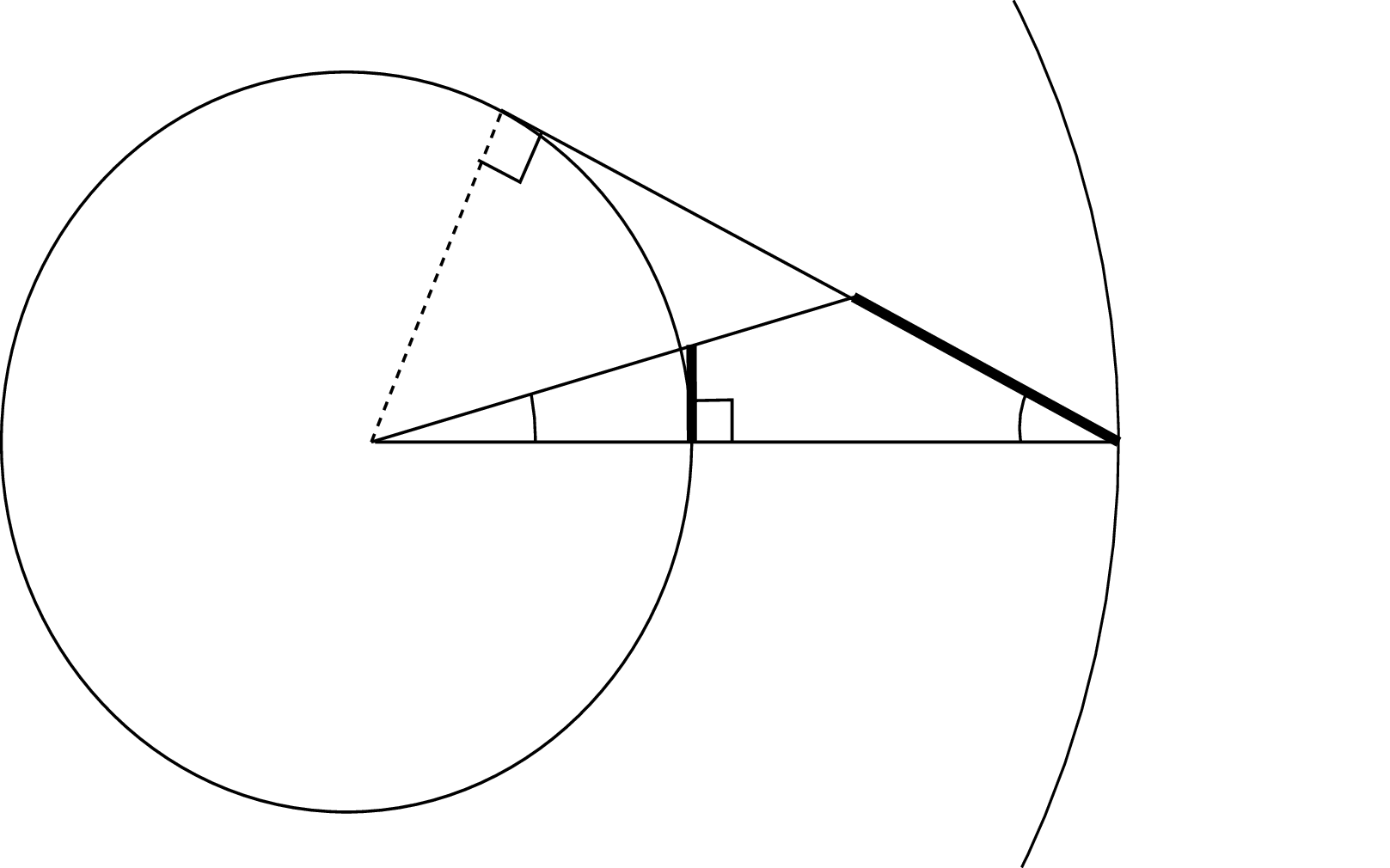
\caption{Extremal case of the projection} \label{fig:proj}
\end{figure}

A straightforward computation shows that $AB= \rho \tan t$, $\displaystyle A'B'=6 \rho \frac{\sin t}{\sin(\theta+t)}$ and $\sin \theta = \frac{1}{6}$ where $\theta=\widehat{B'A'O}$.
Hence, 
\[
\frac{A'B'}{AB}=\frac{6 \cos t}{\sin(\theta+t)} \longrightarrow 36.
\]
We derive the desired Lipschitz constant after pre- and post-composition with a $\lambda$-quasi-isometry.
\end{proof}

\section{Splitting manifolds}

The goal of this section is to prove the following result.

\begin{theorem} \label{theo:split1}
Let $M$ be a complete Riemannian $n$-manifold with nonnegative Ricci curvature.
For every bounded domain~$D$ of~$M$, there exists a smooth hypersurface~$S$ of~$D$ with boundary lying in~$\partial D$ which decomposes~$D$ into two domains $D_1$ and~$D_2$ with $\vol_n(D_i) \geq \alpha_n \vol_n(D)$ such that
\[
\vol_{n-1}(S) \leq A_n \, \vol_n(D)^{\frac{n-1}{n}}
\]
where $A_n$ and $\alpha_n$ are explicit positive constants depending only on~$n$.
\end{theorem}

Before proceeding to the proof of this result, let us recall the Bishop-Gromov inequality, which plays a crucial role in the study of manifolds with nonnegative Ricci curvature.

Let $M$ be a complete Riemannian $n$-manifold with nonnegative Ricci curvature.
Then for every $0<r<R$ and $x \in M$, we have
\begin{equation} \label{eq:bishop}
\frac{\vol_n(B(x,R))}{\vol_n(B(x,r))} \leq \left( \frac{R}{r} \right)^n.
\end{equation}
In particular, the volume of every $R$-ball in~$M$ is at most~$\omega_n R^n$, where $\omega_n$ is the volume of a unit ball in the Euclidean $n$-space.

A consequence of the Bishop-Gromov inequality is the following classical packing estimate.

\begin{lemma} \label{lem:packing}
Let $r>0$.
Every ball of radius $2r$ in a complete Riemannian $n$-manifold~$M$ with nonnegative Ricci curvature can be covered by $c_n=9^n$ balls of radius~$r$.
\end{lemma}

\begin{proof}
Let $B(x,2r)$ be a ball of radius~$2r$ in~$M$.
Consider a maximal system of disjoint $\frac{r}{2}$-balls in~$M$ with centers~$x_i$ lying in~$B(x,2r)$.
Clearly, the balls~$B(x_i,\frac{r}{2})$ are contained in~$B(x,\frac{5}{2} r)$.
From the Bishop-Gromov inequality~\eqref{eq:bishop}, the number of elements in such maximal system of disjoint balls is at most
\[
\frac{\vol_n(B(x,\frac{5}{2} r))}{\vol_n(B(x_i,\frac{r}{2}))} \leq \frac{\vol_n(B(x_i,\frac{9}{2} r))}{\vol_n(B(x_i,\frac{r}{2}))} \leq 9^n.
\]
On the other hand, by maximality, the balls~$B(x_i,r)$ cover~$B(x,2r)$.
\end{proof}

The following result about the existence of capacitors can be found in~\cite[Lemma 2.2]{CM08} (see also~\cite{Kor93,GY99,GNY04}).
We include a proof for the sake of completeness.

\begin{lemma} \label{lem:capa}
Let $D$ be a bounded domain of a complete Riemannian \mbox{$n$-manifold}~$M$ with nonnegative Ricci curvature.
Let $r=  \left(\frac{ \vol_n(D)}{\omega_n (1+2c_n)}\right)^\frac{1}{n}$ where $c_n=9^n$.
Then there exist two domains $A_1$ and~$A_2$ in~$D$ with $\vol_n(A_i) \geq \frac{1}{1+2c_n} \, \vol_n(D)$ at distance at least~$r$ from each other in~$M$.
\end{lemma}

\begin{proof}
Given a subset~$A$ of~$M$, let 
$
\vol_D(A) = \vol_n(A \cap D).
$
Fix $\lambda_n=\frac{1}{1+2 c_n}$.
For $m \in \N^*$, consider the function $\Psi_m:M^m \to \R$ defined on the $m$-fold product~$M^m$ of~$M$ as
\[
\Psi_m(x_1,\cdots,x_m) = \vol_D \left( \bigcup_{i=1}^m B(x_i,r) \right).
\]
Since $D$ is bounded, the function~$\Psi_m$ attains its maximum at some point~$\mathbf{x}_m=(x_{1,m}, \cdots, x_{m,m})$ in~$M^m$.
Clearly, $\max(\Psi_1) = \Psi_1(\mathbf{x}_1) \leq \omega_n r^n=\lambda_n \vol_n(D)$ from the Bishop-Gromov inequality~\eqref{eq:bishop} and $\max(\Psi_m) = \Psi_m(\mathbf{x}_m)=\vol_n(D)$ for $m$ large enough.
Let $m_0 \in \N^*$ be the smallest integer such that $\max(\Psi_{m_0}) = \Psi_{m_0}(\mathbf{x}_{m_0}) \geq \lambda_n \vol_m(D)$.
Define $A_1 = \bigcup_{i=1}^{m_0} B(x_{i,m_0},r) \cap D$.
Note that 
\[
\vol_n(A_1) \geq \lambda_n \vol_n(D).
\]
We also have
\begin{align*}
\vol_n(A_1) & \leq  \vol_D \left( \bigcup_{i=1}^{m_0 -1} B(x_{i,m_0},r) \right) + \vol_D(B(x_{i,{m_0}},r)) \\
 & \leq \max(\Psi_{m_0-1}) + \max(\Psi_1) \\
 & \leq \lambda_n \vol_n(D) + \lambda_n \vol_n(D) = 2 \lambda_n \vol_n(D).
\end{align*}
By Lemma~\ref{lem:packing}, the set $A'_1= \bigcup_{i=1}^{m_0} B(x_{i,m_0},2r) \cap D$ can be covered by $m_0 \, c_n$ balls of radius~$r$ in~$M$.
The union of these $r$-balls decomposes into the union of $c_n=9^n$ sets of the form $\bigcup_{i=1}^{m_0} B(x_i,r)$ with $x_i \in M$.
Since 
\[
\vol_D \left( \bigcup_{i=1}^{m_0} B(x_i,r) \right) \leq \vol_D \left( \bigcup_{i=1}^{m_0} B(x_{i,m_0},r) \right) = \vol_n(A_1),
\]
we derive
\[
\vol_n(A'_1) \leq c_n \, \vol_n(A_1) \leq 2 \lambda_n c_n \, \vol_n(D).
\]
Thus, for $A_2=D \setminus A'_1$, we obtain
\[
\vol_n(A_2) \geq (1-2\lambda_n c_n) \, \vol_n(D) = \lambda_n \vol_n(D).
\]
By construction, we have $d(A_1,A_2) \geq r$.
Note that we can also smooth out the boundary of~$A_i$.
\end{proof}

We can now derive Theorem~\ref{theo:split1}.

\begin{proof}[Proof of Theorem~\ref{theo:split1}]
Let $A_1$ and $A_2$ be as in Lemma~\ref{lem:capa}.
Denote by~$A_1(r)$ the $r$-neighborhood of~$A_1$ in~$D$.
Note that $A_2$ lies in the complementary set of~$A_1(r)$ in~$D$.
Define a Lipschitz function $f:D \to \R$ as
\[
f(x) = 
\left\{
\begin{array}{ll}
1 & \mbox{if } x \in A_1 \\
1-\frac{d(x,A_1)}{r} & \mbox{if } x \in A_1(r) \setminus A_1 \\
0 & \mbox{if } x \notin A_1(r)
\end{array}
\right.
\]
The function~$f$ is not a Morse function.
However, it can be approximated by a Morse function $f_\varepsilon:D \to \R$ with $\varepsilon >0$ such that
\begin{enumerate}
\item[$\bullet$] $f_\varepsilon(x) >1$ if $x \in A_1$;
\item[$\bullet$] $f_\varepsilon(x) <0$ if $x \notin A_1(r)$;
\item[$\bullet$] $|df_\varepsilon | \leq |df|+\varepsilon$.
\end{enumerate}
Hence, from the coarea formula \cite[13.4]{BZ}, we obtain
\begin{align*}
\int_0^1 \vol_{n-1} (f_\varepsilon^{-1}(t)) \, dt & \leq \int_D |df_\varepsilon | \, dv \\
 & \leq \left(\frac{1}{r}+\varepsilon \right) \, \vol_n(D) \\
 & \leq \left( \omega_n (1+2 c_n) \right)^\frac{1}{n} \, \vol_n(D)^\frac{n-1}{n} + \varepsilon \, \vol_n(D).
\end{align*}
To avoid burdening the argument by epsilontics, we will assume that $\varepsilon=0$.
Thus, there exists a regular value $t_0 \in (0,1)$ such that 
\[
\vol_{n-1} (f_\varepsilon^{-1}(t_0)) \leq \left( \omega_n (1+2 c_n) \right)^\frac{1}{n} \, \vol_n(D)^\frac{n-1}{n}.
\]
The preimage $S_\varepsilon=f_\varepsilon^{-1}(t_0)$ is a smooth hypersurface of~$D$ with boundary lying in~$\partial D$ which decomposes~$D$ into two domains $D_1 = f_\varepsilon^{-1}((-\infty,t_0])$ and $D_2=f_\varepsilon^{-1}([t_0,\infty))$. 
Since $A_i \subset D_i$, we obtain 
\[
\vol_n(D_i) \geq \frac{1}{1+2 c_n} \, \vol_n(D).
\]
The theorem follows with $A_n = \left( \omega_n (1+2 c_n) \right)^\frac{1}{n}$ and $\alpha_n = \frac{1}{1+2c_n}$.
\end{proof}

\section{Splitting manifolds preserving the CW complex structure}

In this section, we establish a version of Theorem~\ref{theo:split1} which preserves the CW complex structure of the manifold with quasi-convex cells.

\begin{theorem} \label{theo:split2}
Fix $\rho >0$ and $\lambda>1$ close enough to~$1$.
Let $M$ be a complete Riemannian $n$-manifold with nonnegative Ricci curvature.
Consider a bounded domain~$D$ in~$M$ endowed with a CW complex structure with quasi-convex cells of size~$(\rho,\lambda)$.
Suppose $D$ has more than one $n$-cell.
Then there exists an $(n-1)$-chain~$S$ of~$D$ with boundary lying in~$\partial D$ which decomposes~$D$ into two $n$-dimensional CW subcomplexes~$D_1$ and~$D_2$ with $\vol_n(D_i) \geq \alpha'_n \vol_n(D)$ such that
\begin{equation} \label{eq:split2}
\vol_{n-1}(S) \leq A'_n \, \vol_n(D)^{\frac{n-1}{n}}
\end{equation}
where $A'_n$ and $\alpha'_n$ are explicit positive constants depending only on~$n$.
\end{theorem}

\begin{proof}
For the sake of simplicity, we will assume that $\lambda=1$, that is, the $n$-cells~$\Delta$ of the CW complex decomposition of~$D$ are flat.
The general case only requires technical changes to keep track of the bi-Lipschitz factor.

From Theorem~\ref{theo:split1}, the domain~$D$ decomposes into two domains~$D_1$ and~$D_2$ such that their common boundary~$S=\partial D_1=\partial D_2$ satisfies
\[
\vol_{n-1}(S) \leq A_n \, \vol_n(D)^\frac{n-1}{n}.
\]
Without loss of generality, we can assume that no connected component of~$D_1$ and~$D_2$ lies in a single $n$-cell of the CW complex structure of~$D$.
This can be achieved by translating such connected components if necessary.

\medskip

We want to deform the hypersurface~$S$ into the $(n-1)$-skeleton of the CW complex decomposition of~$D$ while controlling its $(n-1)$-volume and the $n$-volume of the two domains of~$D$ it bounds.
By perturbing the hypersurface~$S$ if necessary, we can assume that it intersects transversally the $(n-1)$-cells of~$D$ and decomposes its $n$-cells into finitely many domains.
Thus, the hypersurface~$S$ decomposes into finitely many pieces~$(S_i)_{i \in I}$ obtained as the connected components of the intersection of~$S$ with the $n$-cells of the CW complex decomposition of~$D$.
For every~$i \in I$, denote by~$\Delta_i$ the $n$-cell of~$D$ in which $S_i$ lies.
By construction, $\partial S_i$ lies in~$\partial \Delta_i$.
We also have
\[
\vol_{n-1}(S) = \sum_{i \in I} \vol_{n-1}(S_i).
\]
Consider a volume-minimizing $(n-1)$-chain~$S'_i$ in~$\Delta_i$ with the same boundary as~$S_i$.

If~$S'_i$ is at distance at least~$\rho$ from the center of~$\Delta_i$, we replace~$S_i$ with the image~$S''_i$ of~$S'_i$ by the projection~$\pi_{\Delta_i}$ given by Lemma~\ref{lem:proj}.
The new $(n-1)$-chain~$S''_i$ lies in~$\partial \Delta_i$, has the same boundary as~$S_i$ and satisfies 
\[
\vol_{n-1}(S''_i) \leq 36^{n-1} \, \vol_{n-1}(S_i)
\]
from Lemma~\ref{lem:proj}.

If some point~$x'_i$ of~$S'_i$ is at distance at most~$\rho$ from the center of~$\Delta_i$, we replace~$S_i$ with a volume-minimizing $(n-1)$-chain~$S''_i$ in~$\partial \Delta_i$ with the same boundary as~$S_i$.
Note that $\vol_{n-1}(S''_i) \leq \frac{1}{2} \vol_{n-1}(S^{n-1}(6 \rho))$.
Since $S'_i$ is a minimal $(n-1)$-chain of~$\Delta_i$ with boundary lying in~$\partial \Delta_i$, the volume of the intersection of~$S'_i$ with the $\rho$-ball of~$\Delta_i$ centered at~$x'_i$ is at least~$\omega_{n-1} \, \rho^{n-1}$ from the monotonicity formula, \cf~\cite[\S9]{mor}.
That is,
\[
\vol_{n-1}(S'_i) \geq \vol_{n-1}(S'_i \cap B(x'_i,\rho)) \geq \omega_{n-1} \, \rho^{n-1}.
\]
Hence,
\[
\vol_{n-1}(S''_i) \leq \tfrac{1}{2} \vol_{n-1}(S^{n-1}(6\rho)) =  \tfrac{1}{2} n \, \omega_n \, 6^{n-1} \, \rho^{n-1} \leq \frac{n \, \omega_n \, 6^{n-1}}{2 \, \omega_{n-1}} \vol_{n-1}(S'_i).
\]

In both case, we obtain an $(n-1)$-chain~$S''_i$ in~$\partial \Delta_i$ with the same boundary as~$S_i$ such that
\[
\vol_{n-1}(S''_i) \leq C_{n-1} \, \vol_{n-1}(S_i)
\]
where $C_{n-1} = \max \{ 36^{n-1}, \frac{n \, \omega_n \, 6^{n-1}}{2 \, \omega_{n-1}}  \}$. 

Replacing~$S_i$ with~$S''_i$ gives rise to an $(n-1)$-chain~$S''=\cup_{i \in I} S''_i$ lying in the $(n-1)$-skeleton of the CW complex decomposition of~$D$, with boundary~$\partial S''$ lying in~$\partial D$.
This $(n-1)$-chain bounds two $n$-dimensional CW subcomplexes~$D''_1$ and~$D''_2$ in~$D$.
Furthermore, the volume of~$S''$ is bounded from above as follows
\begin{equation} \label{eq:S''}
\vol_{n-1}(S'') \leq C_{n-1} \, \vol_{n-1}(S) \leq C_{n-1} A_n \, \vol_n(D)^\frac{n-1}{n}.
\end{equation}
Hence, the bound~\eqref{eq:split2} for $A'_n=\max \{C_{n-1} A_n, n \, \omega_n^\frac{1}{n} \, 3^{n-1} \}$. 
(The reason for the second argument in the max will become clear later.)

\medskip

It remains to bound from below the volume of the two CW subcomplexes $D''_1$ and~$D''_2$ bounded by~$S''$.
By the isoperimetric inequality, the volume of the region~$R_i$ of~$\Delta_i$ bounded by~$S_i$ and~$S''_i$ satisfies
\begin{align*}
\vol_n(R_i) & \leq I_n \left( \vol_{n-1}(S_i) + \vol_{n-1}(S''_i) \right)^\frac{n}{n-1} \\
 & \leq I_n (1+C_{n-1})^\frac{n}{n-1} \, \vol_{n-1}(S_i)^\frac{n}{n-1},
\end{align*}
where $I_n=\left( \frac{1}{n^n \omega_n} \right)^\frac{1}{n-1}$ is the $n$-dimensional isoperimetric constant in~$\R^n$. \\ 
Since $R_i \subset \Delta_i \subset B^n(6 \rho)$, we also have
\[
\vol_n(R_i) \leq \vol_n(\Delta_i) \leq \vol_n(B^n(6 \rho)) = \omega_n \, 6^n \, \rho^n.
\]

Now, let $I_-$ be the set of all the indices~$i \in I$ such that $\vol_{n-1}(S_i) \leq \rho^{n-1}$ and $I_+=I \setminus I_-$ be the set of all the indices~$i \in I$ such that $\vol_{n-1}(S_i) > \rho^{n-1}$.
By substracting up the volume change of~$D''_j \cap \Delta_i$ through the replacement process for every $n$-cell~$\Delta_i$ of the CW complex decomposition of~$D$, we obtain
\begin{align*}
\vol_n(D''_j) & \geq \vol_n(D_j) - \sum_{i \in I_-} I_n (1+C_{n-1})^\frac{n}{n-1} \, \vol_{n-1}(S_i)^\frac{n}{n-1} - \sum_{i \in I_+}  \omega_n \, 6^n \, \rho^n \\
 & \geq \vol_n(D_j) - \sum_{i \in I_-} I_n (1+C_{n-1})^\frac{n}{n-1} \rho \, \vol_{n-1}(S_i) - \sum_{i \in I_+} \omega_n  \, 6^n \, \rho \, \vol_{n-1}(S_i) \\
 & \geq \vol_n(D_j) - K_n \, \rho \, \vol_{n-1}(S)
\end{align*}
where $K_n = \max \{ I_n(1+C_{n-1})^\frac{n}{n-1},\omega_n \, 6^n\}$. 

\medskip

Suppose the number~$N$ of $n$-cells in the CW complex decomposition of~$D$ is greater or equal to~$N_0=\frac{1}{\omega_n} \left( \frac{K_n A'_n}{\alpha_n} \right)^n$.
From the volume estimate~\eqref{eq:N}, we have $\rho \leq \frac{1}{2} \left( \frac{\vol_n(D)}{N_0 \omega_n} \right)^\frac{1}{n} = \frac{\alpha_n}{2K_n A'_n} \, \vol_n(D)^\frac{1}{n}$.
Combined with the bound~\eqref{eq:S''}, this implies that 
\[
K_n \, \rho \, \vol_{n-1}(S'') \leq \tfrac{\alpha_n}{2} \, \vol_n(D) \leq \tfrac{1}{2} \, \vol_n(D_j).
\]
Thus, 
\[
\vol_n(D''_j) \geq \tfrac{1}{2} \, \vol_n(D_j) \geq \tfrac{\alpha_n}{2} \, \vol_n(D).
\]

If $N$ is less than $N_0$, we can take for $D''_1$ any $n$-cell of~$D$ and for $D''_2$ the union of the remaining $n$-cells.
In this case, still from~\eqref{eq:N}, we deduce that
\[
\vol_{n-1}(\partial D''_j) \leq \vol_{n-1}(S^{n-1}(6 \rho)) = n \, \omega_n \, 6^{n-1} \, \rho^{n-1} \leq n \, \omega_n^\frac{1}{n} \, 3^{n-1} \, \vol_n(D)^\frac{n-1}{n}
\]
where $S^{n-1}(6 \rho)$ is the $6 \rho$-sphere centered at the origin of~$\R^n$.
By our choice of~$A'_n$, the bound~\eqref{eq:split2} is still satisfied.
We also derive from~\eqref{eq:N} that
\[
\vol_n(D''_j) \geq \vol_n(B^n(\rho)) = \omega_n \, \rho^n \geq \frac{1}{6^n N_0} \vol_n(D) = \omega_n \left( \frac{\alpha_n}{6 K_n A'_n} \right)^n \vol_n(D).
\]

In both cases, the result follows by renaming~$D''_j$ into~$D_j$ with~$\alpha'_n = \min \left\{\frac{\alpha_n}{2}, \omega_n \left( \frac{\alpha_n}{6 K_n A'_n} \right)^n \right\}$.
\end{proof}

\begin{remark}
By the simplicial approximation theorem, we can assume that $S$ is an $(n-1)$-dimensional simplicial chain of a simplicial subdivision of~$D$ with boundary lying in~$\partial D$ which decomposes~$D$ into two $n$-dimensional simplicial subcomplexes $D_1$ and~$D_2$ satisfying the same geometric estimates.
\end{remark}

\section{Merging sweep-outs}

For the next result, we need to introduce a min-max value for sweep-outs defined in terms of PL Morse functions.

\begin{definition} 
Let $D$ be a piecewise flat simplicial $n$-complex homeomorphic to a compact $n$-manifold with (possibly empty) boundary.
A function $f:D \to \R$ is a \emph{PL Morse function} if it is linear in restriction to every simplex of~$D$ and takes pairwise distinct values at the vertices of~$D$.
Note that a PL Morse function is uniquely determined by its values on the set of vertices of~$D$.

\medskip

The $(n-1)$-volume width of a PL Morse function~$f:D \to \R$ is defined as
\[
\A(f) = \sup_t \vol_{n-1}(f^{-1}(t)).
\]
Similarly, the $(n-1)$-volume width of~$D$ is defined as
\[
\A(D) = \inf_f \A(f) = \inf_f \sup_t \vol_{n-1}(f^{-1}(t))
\]
where $f$ runs over all PL Morse functions on some simplicial subdivision of~$D$.

\medskip

When $D$ is a compact Riemannian $n$-manifold with (possibly empty) piecewise smooth boundary and no fixed underlying simplicial complex structure, we extend the notion of PL Morse function as follows.
A function $f:D \to \R$ is a \emph{(generalized) PL Morse function} if there exists a simplicial complex structure on~$D$ for which $f$ is a PL Morse function.
\end{definition}

\forget
\begin{definition} \label{def:width}
Let $D$ be a compact Riemannian $n$-manifold with (possibly empty) piecewise smooth boundary.
The $(n-1)$-volume width of a Morse function~$f:D \to \R$ is defined as
\[
\A(f) = \sup_t \vol_{n-1}(f^{-1}(t)).
\]
Similarly, the $(n-1)$-volume width of~$D$ is defined as
\[
\A(D) = \inf_f \A(f) = \inf_f \sup_t \vol_{n-1}(f^{-1}(t))
\]
where $f$ runs over all Morse functions on~$D$.
\end{definition}
\forgotten

The following estimate on the $(n-1)$-volume width of a bounded domain in terms of the $(n-1)$-volume widths of the subdomains it is made of results from a cut-and-paste argument.

\begin{proposition} \label{prop:merging}
Let $D$ be a piecewise flat simplicial $n$-complex homeomorphic to a compact $n$-manifold with (possibly empty) boundary.
Let $S$ be an $(n-1)$-chain of~$D$ with boundary lying in~$\partial D$ which decomposes~$D$ into two $n$-dimensional simplicial subcomplexes~$D_1$ and~$D_2$.
Then 
\[
\A(D) \leq \max \{\A(D_1),\A(D_2) \} + 2n \, \vol_{n-1}(S).
\]
\end{proposition}

\begin{remark}
With extra care in the proof, it seems likely that the $2n$ factor can be dropped in the conclusion of Proposition~\ref{prop:merging}.
\end{remark}

\begin{proof}
The boundary of the $(n-1)$-chain~$S$ decomposes~$\partial D$ into two simplicial subcomplexes~$S_1$ and~$S_2$ with~$S_i$ lying in~$\partial D_i$.
Therefore, the boundary of~$D_i$ decomposes into~$\partial D_i = S \cup S_i$.

Let $f_i:D_i \to \R$ be a PL Morse function defined on a simplicial subdivision of~$D_i$.
By adding constants to $f_1$ and~$f_2$ if necessary, we can assume that 
\[
\max_{D_1} f_1 < -1 < 1 < \min_{D_2} f_2.
\]

We want to construct a PL Morse function $g:D \to \R$ defined on a simplicial subdivision of~$D$ by deforming $f_1$ and~$f_2$ in the neighborhood of~$S$ into two other PL Morse functions which agree on~$S$.
The new PL Morse function $g:D \to \R$ should satisfy
\[
\vol_{n-1}(g^{-1}(t)) \leq \vol_{n-1}(f_i^{-1}(t)) + \vol_{n-1}(S) + \delta
\]
for every $t \in \R$, where $\delta$ is an arbitrarily small error term.

Let $\varepsilon \in (0,1)$ small enough.
From~\cite[p.~122]{spa}, we can assume that the simplicial subdivisions on~$D_1$ and~$D_2$ arise from the same simplicial subdivision of~$D$.
The $n$-simplices of~$D_i$ meeting~$S$ form a (closed) neighborhood~$\mathcal{U}_i$ of~$S$ in~$D_i$.
Let $e$ be an edge of this neighborhood with exactly one endpoint, say~$p$, lying in~$S$.
Denote by~$p_e$ the point of~$e$ such that $d(p,p_e) = \varepsilon \, \length(e)$.
Note that for $\varepsilon$ small enough, we can assume that the values taken by~$f_i$ at the points~$p_e$ are pairwise distinct.

Now, we truncate every $n$-simplex~$\sigma$ of~$\mathcal{U}_i$ along the hyperplane passing through the points~$p_e$ where $e$ is an edge of~$\sigma$.
This gives rise to a CW complex structure of~$\mathcal{U}_i$ (finer than the one given by the initial triangulation) whose $n$-cells are convex Euclidean polyhedra, namely truncated $n$-simplices.

Denote by~$\mathcal{V}_i$ the (closed) neighborhood of~$S$ in~$\mathcal{U}_i$ formed of the $n$-cells (\ie, truncated $n$-simplices) of~$D_i$ meeting~$S$.
Every $n$-cell of~$\mathcal{V}_i$ meets~$S$ along a $k$-simplex with $k \leq n-1$.
An $n$-cell~$\Delta$ of~$\mathcal{V}_i$ is said to be \emph{big} if it meets~$S$ along an $(n-1)$-simplex~$\Delta \cap S$ and \emph{small} otherwise.
Observe that every hyperplane intersects a small $n$-cell (viewed as a convex Euclidean polyhedron) of~$\mathcal{V}_i$ along a region of $(n-1)$-volume at most~$\delta_\varepsilon$, where $\delta_\varepsilon \to 0$ when $\varepsilon$ goes to zero.
Similarly, every hyperplane intersects a big $n$-cell~$\Delta$ of~$\mathcal{V}_i$ along a region of $(n-1)$-volume at most~$\eta_\varepsilon \, \vol_{n-1}(\Delta \cap S)$, where $\eta_\varepsilon \to 1$ when $\varepsilon$ goes to zero.

Without introducing new vertices, we define a new triangulation of~$D_i$ by subdividing the $n$-cells (convex Euclidean polyhedra) of~$\mathcal{U}_i$ into $n$-simplices.
In the process, big $n$-cells (which have exactly $2n$ vertices) are split into $n$ simplices of maximal dimension and small $n$-cells (which have at most $n^2$ vertices) are split into at most $n^2$ simplices of maximal dimension.

We define a PL Morse function $g:D \to \R$ with respect to the new triangulation of~$D$ which agrees with $f_i$ on the vertices of~$D_i \setminus S$ and takes the vertices of~$S$ to pairwise distinct values close to zero in~$(-1,1)$.

The level sets of~$g$ are transverse to the $(n-1)$-faces of the triangulation of~$D$ and intersect each $n$-simplex of~$\mathcal{V}=\mathcal{V}_1 \cup \mathcal{V}_2$ along a (possibly empty) hyperplane.
The volume estimates on these hyperplanes and the bounds on the maximal number of $n$-simplices in an $n$-cell yield the following estimate
\begin{align}
\vol_{n-1}(g^{-1}(t) \cap \mathcal{V}) & \leq \sum_{\sigma \mbox{\footnotesize{ big}}} \vol_{n-1}(g^{-1}(t) \cap \sigma) + \sum_{\sigma \mbox{\footnotesize{ small}}} \vol_{n-1}(g^{-1}(t) \cap \sigma) \nonumber \\
 & \leq \sum_{\sigma \mbox{\footnotesize{ big}}} n \, \eta_\varepsilon \, \vol_{n-1}(S \cap \sigma) + \sum_{\sigma \mbox{\footnotesize{ small}}} n^2 \, \delta_\varepsilon \nonumber \\
 & \leq 2n \, \eta_\varepsilon \, \vol_{n-1}(S) + n^2 N_s \, \delta_\varepsilon \label{eq:level} \\
 & \lesssim 2n \, \vol_{n-1}(S) \nonumber
\end{align}
where the first sum is over the $n$-simplices lying in a big $n$-cell of~$\mathcal{V}$, the second sum is over the $n$-simplices lying in a small $n$-cell of~$\mathcal{V}$ and $N_s$ is the number of $n$-simplices lying in a small $n$-cell of~$\mathcal{V}$.

The level sets of~$g$ satisfy the following properties for different values of~$t$.
For $t <-1$, the level set $g^{-1}(t) \setminus \mathcal{V}$ lies in~$D_1$ and agrees with $f_1^{-1}(t) \setminus \mathcal{V}_1$, while for $t >1$, it lies in~$D_2$ and agrees with $f_2^{-1}(t) \setminus \mathcal{V}_2$.
Finally, for $t \in [-1,1]$, the level set $g^{-1}(t)$ lies in~$\mathcal{V}$.

From this, we immediately derive
\[
\vol_{n-1}(g^{-1}(t)) \leq \max \{ \vol_{n-1}(f_1^{-1}(t)), \vol_{n-1}(f_2^{-1}(t)) \} + \vol_{n-1}(g^{-1}(t) \cap \mathcal{V}).
\]
Combined with~\eqref{eq:level}, we obtain
\[
\A(g) \leq \max \{ \A(f_1), \A(f_2) \} + 2n \, \eta_\varepsilon \, \vol_{n-1}(S) + n^2 N \, \delta_\varepsilon.
\]
The desired result follows by letting $\varepsilon$ go to zero, namely
\[
\A(D) \leq \max \{ \A(D_1), \A(D_2) \} + 2n \, \vol_{n-1}(S).
\]
\end{proof}

\forget
\begin{proof}
The boundary of the $(n-1)$-chain~$S$ decomposes~$\partial D$ into two parts~$S_1$ and~$S_2$ with~$S_i$ lying in~$\partial D_i$.
Therefore, the boundary of~$D_i$ decomposes into~$\partial D_i = S \cup S_i$.

Let $f_i:D_i \to \R$ be a Morse function.
We want to deform~$f_i$ in the neighborhood of~$S$ into another Morse function~$g_i:D_i \to [0,1]$ such that the level set~$g_i^{-1}(1)$ agrees with~$S$ and
\[
\vol_{n-1}(g_i^{-1}(t)) \leq \vol_{n-1}(f_i^{-1}(t)) + \vol_{n-1}(S) + \delta
\]
for every $t \in [0,1]$, where $\delta$ is an arbitrarily small error term.

Without loss of generality, we can assume that $f_i$ has no critical point in the neighborhood of~$\partial D_i$, that its level sets transversally meet~$\partial D_i$ and that the range of~$f_i$ lies between $0$ and~$1$.
Consider a (closed) neighborhood of~$S$ in~$D_i$ quasi-isometric to~$S \times [0,\varepsilon]$ with $\varepsilon>0$ small enough.
Up to reparametrisation, we will identify this neighborhood with~$S \times [0,1]$ and $S$ with~$S \times \{1\}$, \cf~Figure~\ref{fig:boundary}.

\begin{figure}[htb] 
\caption{Level sets in the neighborhood of~$S$} \label{fig:boundary}
\end{figure}

Now, let us define $g_i:D_i \to [0,1]$ as follows: outside~$S \times [0,1]$, the function~$g_i$ agrees with~$f_i$, while on~$S \times [0,1]$, it is given by
\[
g_i(s,\tau) = \left\{
\begin{array}{cl}
f_i(s) & \mbox{if } \tau \leq f_i(s) \\
\tau & \mbox{if } \tau > f_i(s)
\end{array}
\right.
\]
for every $(s,\tau) \in S \times [0,1]$.
Since $f_i$ is nonnegative, the functions~$f_i$ and~$g_i$ agree on~$S \times \{0\}$.
Furthermore, by construction, we have
\[
g_i^{-1}(t) = (f_i)^{-1}_{D_i \setminus S \times [0,1]}(t) \cup \left( (f_i)^{-1}_{S \times \{0\}}(t) \times [0,t] \right) \cup \left( (f_i)^{-1}_{S \times \{0\}}([0,t]) \times \{ t \} \right)
\]
for every $t \in [0,1]$, \cf~Figure~\ref{fig:boundary}.

By taking $\varepsilon$ sufficiently small, the $(n-1)$-volume of the second term of~$g_i^{-1}(t)$, namely~$(f_i)^{-1}_{S \times \{0\}}(t) \times [0,t]$, can be made arbitrarily small.
Similarly, the $(n-1)$-volume of the last term of~$g_i^{-1}(t)$, namely \mbox{$(f_i)^{-1}_{S \times \{0\}}([0,t]) \times \{ t \}$}, is bounded from above by the $(n-1)$-volume of~$S$ up to an arbitrarily small error term.
Now, we can perturb~$g_i$ into a Morse function, still denoted by~$g_i$, with values in~$[0,1]$, without increasing too much the $(n-1)$-volume of its level sets while making sure that $S=g_i^{-1}(1)$.
By construction, we have 
\[
\A(g_i) \leq \A(f_i) + \vol_{n-1}(S) + \delta.
\]

\medskip

Let $f:D \to \R$ be a Morse function.
Consider the functions~$g_i:D_i \to [0,1]$ previously constructed from the restrictions~$f_i$ of~$f$ to~$D_i$.
Note that both $g_1$ and~$g_2$ (and so $2-g_2$) are constant equal to~$1$ on~$S$.
Define $g:D \to \R$ as
\[
g(x) = \left\{
\begin{array}{ll}
g_1(x) & \mbox{if } x \in D_1 \\
2-g_2(x) & \mbox{if } x \in D_2
\end{array}
\right.
\]
Slightly perturbing~$g$ in the neighborhood of~$S$, we can assume that it is a Morse function.
By construction, we have
\[
\A(g) \leq \A(f) + \vol_{n-1}(S) + \delta
\]
for some arbitrarily small error term~$\delta$, which yields the desired result.
\end{proof}
\forgotten

\section{Main theorem}

We can now prove the main theorem of this article.

\begin{theorem} \label{theo:main}
Let $M$ be a complete Riemannian $n$-manifold with nonnegative Ricci curvature.
For every bounded domain~$D$ in~$M$, there exists a PL Morse function $f:D \to \R$ such that
\[
\sup_t \vol_{n-1}(f^{-1}(t)) \leq C_n \, \vol_n(D)^\frac{n-1}{n}
\]
where $C_n$ is an explicit positive constant depending only on~$n$.
\end{theorem}

\begin{remark}
When $M$ is a closed Riemannian $n$-manifold and $D$ agrees with the whole manifold~$M$, we recover Theorem~\ref{theo:A} from the introduction.
\end{remark}

\begin{proof}[Proof of Theorem~\ref{theo:main}]
Fix $\rho >0$ small enough so that there exists a CW complex structure on~$D$ with quasi-convex cells of size~$(\rho,\lambda)$ with~$\lambda>1$ close enough to~$1$, \cf~Proposition~\ref{prop:existence}.
As in the proof of Theorem~\ref{theo:split2}, we will assume that $\lambda=1$ for the sake of simplicity.
That is, the quasi-convex $n$-cells of the CW complex structure of~$D$ are flat.
In particular, $D$ is a piecewise flat simplicial $n$-complex.
We argue by induction on the number of quasi-convex $n$-cells in the CW complex structure of~$D$. 

\medskip

Suppose that $D$ is composed of a single $n$-cell.
Every polyhedron lying in~$B^n(6 \rho)$ can be swept out by parallel hyperplanes of volume at most $\vol_{n-1}(B^{n-1}(6 \rho))=\omega_{n-1} 6^{n-1} \rho^{n-1}$, which correspond to the level sets of some PL Morse function.
Now, the volume of~$D$ is greater or equal to $\vol_n(B^n(2\rho))=\omega_n 2^n \rho^n$.
Thus, 
\[
\A(D) \leq 6^{n-1} \omega_{n-1} \, \omega_n^{-\frac{n-1}{n}} \, \vol_n(D)^\frac{n-1}{n}.
\]

\medskip

In the general case, from Theorem~\ref{theo:split2}, there exists an $(n-1)$-chain~$S$ of~$D$ with boundary lying in~$\partial D$ which decomposes~$D$ into two CW subcomplexes~$D_1$ and~$D_2$ with $\vol_n(D_i) \geq \alpha'_n \, \vol_n(D)$ such that
\begin{equation} \label{eq:S}
\vol_{n-1}(S) \leq A'_n \, \vol_n(D)^{\frac{n-1}{n}},
\end{equation}

Both $D_1$ and~$D_2$ have fewer $n$-cells in their CW complex decomposition than~$D$.
Therefore, we can assume that
\[
\frac{\A(D_i)}{\vol_n(D_i)^\frac{n-1}{n}} \leq \frac{\A(D)}{\vol_n(D)^\frac{n-1}{n}}
\]
otherwise the result follows by induction.

Thus, since $\vol_n(D_i) \leq (1-\alpha'_n) \vol_n(D)$, we deduce that
\[
\A(D_i) \leq (1-\alpha'_n)^\frac{n-1}{n} \, \A(D).
\]
Combined with Proposition~\ref{prop:merging} and the bound~\eqref{eq:S}, we obtain
\begin{align*}
\A(D) & \leq \frac{2n}{1-(1-\alpha'_n)^\frac{n-1}{n}} \, \vol_{n-1}(S) \\
& \leq \frac{2n \, A'_n}{1-(1-\alpha'_n)^\frac{n-1}{n}} \, \vol_n(D)^\frac{n-1}{n}
\end{align*}
Hence, the theorem holds with $C_n = \max \left\{\frac{2n  A'_n}{1-(1-\alpha'_n)^\frac{n-1}{n}}, 6^{n-1} \omega_{n-1} \, \omega_n^{-\frac{n-1}{n}} \right\}$.
\end{proof}

Given a closed Riemannian $n$-manifold~$M$, we deduce from the min-max principle of~\cite{DT} the existence of an embedded closed minimal hypersurface in~$M$ with a singular set of Hausdorff dimension at most~$n-8$ and volume at most~$\A(M)$.
Combined with Theorem~\ref{theo:A}, this yields Theorem~\ref{theo:0}. \\

Similarly, Corollary~\ref{coro:decomp} is a consequence of Theorem~\ref{theo:A}.

\begin{proof}[Proof of Corollary~\ref{coro:decomp}]
From Theorem~\ref{theo:A}, a closed Riemannian manifold~$M$ with nonnegative Ricci curvature and dimension~$n \geq 3$ decomposes into two non-necessarily connected domains $D_1$ and~$D_2$ with the same volume such that $\vol_{n-1}(\partial D_i) \leq C_n \, \vol_n(M)^{\frac{n-1}{n}}$.
Slightly perturbing the domains~$D_i$ and connecting their connected components with thin tubes if necessary, we can assume that the domains~$D_i$ are connected.
In dimension two, it follows from~\cite{CC} that any two-sphere with nonnegative curvature can be swept out by a one-parameter family of disjoint simple loops of length at most the length of the shortest closed geodesic.
Since the length of the shortest closed geodesic is bounded in terms of the area of the sphere, we immediately obtain the desired decomposition.
\end{proof}

\section{A curvature-free counter-example}

In this section, we show that the curvature condition in Theorem~\ref{theo:A} cannot be dropped in general.

\medskip

Before going further, we need to briefly review the Almgren-Pitts min-max principle.
For more details, we refer the reader to~\cite{pitts,SS81,CD03,DT}.

Let $M$ be a closed Riemannian $n$-manifold.
Fix $\kk=\Z$ if $M$ is orientable and $\kk=\Z/2\Z$ otherwise.
Denote by~$\ZZ_k(M)$ the $k$-cycle space of~$M$ over~$\kk$, that is, the space of $k$-dimensional integral currents of~$M$ over~$\kk$ with zero boundary.
The mass of a $k$-cycle~$z \in \ZZ_k(M)$ defined as
\[
\M(z) = \sup \left\{ \int_z \omega \mid \omega \mbox{ smooth } \mbox{k-form on } M \mbox{ with } || \omega || \leq 1 \right\}
\]
extends the notion of volume for $k$-submanifolds.

The homotopy groups of the $k$-cycle space endowed with the flat norm topology have been determined by F.~Almgren \cite{alm,pitts}.
More precisely, there is a natural isomorphism between $H_p(\ZZ_k(M);\kk),\{0\})$ and $H_{p+k}(M;\kk)$.
In particular, 
\[
\pi_1(\ZZ_{n-1}(M;\kk),\{0\}) \simeq H_n(M;\kk) \simeq \kk.
\]
This isomorphism allows us to apply the Almgren-Pitts min-max principle to the $(n-1)$-cycle space of~$M$ as follows.

A sweep-out of~$M$ by~$(n-1)$-cycles is defined as a one-parameter family of $(n-1)$-cycles~$(z_t)_{0 \leq t \leq 1}$ starting and ending at the null-cycle which induces a generator of~$\pi_1(\ZZ_{n-1}(M;\kk),\{0\})$.
An example of sweep-out is given by the level sets~$f^{-1}(t)$ of a PL Morse function $f:M \to [0,1]$.

The $(n-1)$-volume width of~$M$ defined in terms of $(n-1)$-cycles is given by the following min-max value
\[
\mathcal{W}_{\ZZ_{n-1}}(M) = \inf_{(z_t)} \sup_{0 \leq t \leq 1} \M(z_t)
\]
where $(z_t)$ runs over the sweep-outs of~$M$.
Clearly, 
\[
\mathcal{W}_{\ZZ_{n-1}}(M) \leq \A(M).
\]

From~\cite{BS10}, Theorem~\ref{theo:A} holds true for every Riemannian two-sphere without any curvature assumption.
It actually holds true for every closed Riemannian surface without any curvature assumption if one allows the multiplicative constant to depend on the genus of the surface (and examples show the multiplicative constant has to depend on the genus).

On the other hand, L.~Guth~\cite[\S5]{guth07} deduced from a construction of~\cite{BI} that every $n$-sphere has a Riemannian metric with unit $n$-volume and $\mathcal{W}_{\ZZ_{n-1}}$ arbitrarily large 
for $n \geq 3$.

The next proposition shows this result extends to every closed manifold of dimension at least three.

\begin{proposition} \label{prop:cex}
Let $M$ be a closed $n$-manifold with $n \geq 3$.
There exists a Riemannian metric on~$M$ with unit $n$-volume and arbitrarily large $(n-1)$-volume width.
In particular, every PL Morse function on~$M$ has a fiber of arbitrarily large $(n-1)$-volume for this metric.
\end{proposition}

\begin{proof}
Consider an $n$-sphere~$S^n$ endowed with a Riemannian metric with unit $n$-volume and $\mathcal{W}_{\ZZ_{n-1}}(S^n)$ arbitrarily large. 
Given a domain~$D$ of~$M$ diffeomorphic to a ball, we consider a degree one map $\varphi: M \to S^n$ taking $D$ to some point~$p \in S^n$ and $M \setminus D$ to $S^n \setminus \{p\}$.
We also define a metric on~$M$ such that the volume of $M \setminus D$ is arbitrarily small and the restriction of~$\varphi$ to~$D$ is an isometry onto~$S^n \setminus \{ p \}$.
Now, since the map~$\varphi$ has degree one, the image by~$\varphi$ of a sweep-out~$(z_t)$ of~$M$ by $(n-1)$-cycles induces a sweep-out of~$S^n$.
From the geometry of~$S^n$, one of the $(n-1)$-cycles of this sweep-out of~$S^n$ has an arbitrarily large mass.
As the map~$\varphi$ is distance, and so mass, nonincreasing, the same holds for the $(n-1)$-cycles~$(z_t)$.
This shows that $\mathcal{W}_{\ZZ_{n-1}}(M)$ can be arbitrarily large while $M$ has unit $n$-volume.
\end{proof}

\begin{remark}
Working with level sets does not allow us to conclude in the proof of Proposition~\ref{prop:cex}.
In order to connect $M$ to the sphere~$S^n$, we need the flexibility provided by the general Almgren-Pitts min-max principle defined on the $(n-1)$-cycle space.
\end{remark}

\end{document}